\documentclass[12pt]{amsproc}
\newtheorem{theorem}{\sc Theorem}[section]
\newtheorem{lemma}[theorem]{\sc Lemma}
\newtheorem{proposition}[theorem]{\sc Proposition}

\begin{document}

\keywords{Profinite groups, Automorphisms, Centralizers}
\subjclass[2010]{20E18, 20E25, 20F40}

\thanks{This work was supported by the Conselho Nacional de Desenvolvimento Cient\'{\i}fico e Tecnol\'ogico (CNPq), Brazil. }

\title[Profinite groups and coprime automorphisms]{Profinite groups and the fixed points\\ of coprime automorphisms}
\author{Cristina Acciarri} 
\address{Cristina Acciarri:  Department of Mathematics, University of Brasilia,
Brasilia-DF, 70910-900 Brazil}
\email{acciarricristina@yahoo.it}

\author{Pavel Shumyatsky} 

\address{Pavel Shumyatsky: Department of Mathematics, University of Brasilia,
Brasilia-DF, 70910-900 Brazil}

\email{pavel@unb.br}

\begin{abstract} The main result of the paper is the following theorem. Let $q$ be a prime and $A$ an elementary abelian group of order $q^3$. Suppose that $A$ acts coprimely on a profinite group $G$ and assume that $C_G(a)$ is locally nilpotent for each $a\in A^{\#}$. Then the group $G$ is locally nilpotent.
\end{abstract}

\maketitle

\section{Introduction}

Let $A$ be a  finite group acting on a finite group $G$. Many well-known results show that the structure of the centralizer $C_G(A)$ (the fixed-point subgroup) of $A$ has influence over the structure of $G$. The influence is especially strong if $(|A|,|G|)=1$, that is, the action of $A$ on $G$ is coprime. Let $A^\#$ denote the set of non-identity elements of A. The following theorem was proved in \cite{ward}.

\begin{theorem}\label{wa} Let $q$ be a prime and $A$ an elementary abelian $q$-group of order at least $q^3$. Suppose that $A$ acts coprimely on a finite group $G$ and assume that $C_{G}(a)$ is nilpotent for each $a\in A^{\#}$. Then $G$ is nilpotent.
\end{theorem}  

There are well-known examples that show that the above theorem fails if the order of $A$ is $q^2$. Indeed, let $p$ and $r$ be odd primes and $H$ and $K$ the groups of order $p$ and $r$ respectively. Denote by $A=\langle a_1,a_2\rangle$ the noncyclic group of order four with generators $a_1,a_2$ and by $Y$ the semidirect product of $K$ by $A$ such that $a_1$ acts on $K$ trivially and $a_2$ takes every element of $K$ to its inverse. Let $B$ be the base group of the wreath product $H\wr Y$ and note that $[B,a_1]$ is normal in $H\wr Y$. Set $G=[B,a_1]K$. The group $G$ is naturally acted on by $A$ and $C_G(A)=1$. Therefore $C_{G}(a)$ is abelian for each $a\in A^{\#}$. But, of course, $G$ is not nilpotent.

In \cite{eu2001} the situation of Theorem \ref{wa} was studied in greater detail and the following result was obtained.

\begin{theorem}\label{2001} Let $q$ be a prime and $A$ an elementary abelian $q$-group of order at least $q^3$. Suppose that $A$ acts coprimely on a finite group $G$ and assume that $C_{G}(a)$ is nilpotent of class at most $c$ for each $a\in A^{\#}$. Then $G$ is nilpotent and the class of $G$ is bounded by a function depending only on $q$ and $c$.
\end{theorem}

Of course, the above results have a bearing on profinite groups. By an automorphism of a profinite group we always mean a continuous automorphism. A group $A$ of automorphisms of a profinite group $G$ is coprime if $A$ has finite order while $G$ is an inverse limit of finite groups whose orders are relatively prime to the order of $A$. Using the routine inverse limit argument it is easy to deduce from Theorem \ref{wa} and Theorem \ref{2001} that if $G$ is a profinite group admitting a coprime group of automorphisms $A$ of order $q^3$ such that $C_G(a)$ is pronilpotent for all $a\in A^{\#}$, then $G$ is pronilpotent; and if $C_G(a)$ is nilpotent for all $a\in A^{\#}$, then $G$ is nilpotent. Yet, certain results on fixed points in profinite groups cannot be deduced from corresponding results on finite groups. The purpose of the present paper is to establish the following theorem.

\begin{theorem}\label{main} Let $q$ be a prime and $A$ an elementary abelian $q$-group of order at least $q^3$. Suppose that $A$ acts coprimely on a profinite group $G$ and assume that $C_{G}(a)$ is locally nilpotent for each $a\in A^{\#}$. Then $G$ is locally nilpotent.
\end{theorem}

Recall that a group is locally nilpotent if every finitely generated subgroup is nilpotent. Though Theorem \ref{main} looks similar to Theorems \ref{wa} and \ref{2001}, in fact it cannot be deduced directly from those results. Moreover, the proof of Theorem \ref{main} is very much different from those of Theorems \ref{wa} and \ref{2001}. In particular, unlike the other results,  Theorem \ref{main} relies heavily on the Lie-theoretical techniques created by Zelmanov in his solution of the restricted Burnside problem \cite{ze1,ze2}. The general scheme of the proof of Theorem \ref{main} is similar to that of the result in \cite{khushu}.

\section{Preparatory work}

Throughout the paper we use without special references the well-known properties of coprime actions:  

\begin{lemma}\label{zizi} If a group $A$ acts coprimely on a finite group $G$, then $ C_{G/N}(A)=C_G(A)N/N$ for any $A$-invariant normal subgroup $N$. 
\end{lemma}

\begin{lemma}\label{76} If $A$ is a noncyclic abelian group acting coprimely on a finite group $G$, then $G$ is generated by the subgroups $C_G(B)$, where $A/B$ is cyclic.
\end{lemma}

The above results both easily extend to the case of coprime automorphisms of profinite groups (see for example \cite[Lemma 3.2]{98}).
Let $x,y$ be elements of a group, or a Lie algebra. We define inductively
$$[x,{}_0y]=x\,\, \text{and}\,\, [x,{}_n\,y]=[[x,{}_{n-1}y],y] \quad \text{for}\,\, n\geq 1.$$

Let $L$ be a Lie algebra. An element $a\in L$ is called ad-nilpotent if there exists a positive integer $n$ such that $[x,{}_n\,a]=0$ for all $x\in L$. Let $X\subseteq L$ be any subset of $L$. By a commutator in elements of $X$ we mean any element of $L$ that can be obtained as a Lie product of elements of $X$ with some system of brackets. The next theorem is due to Zelmanov (see \cite{ze3} or \cite{ze4}).

\begin{theorem}\label{21} Let $L$ be a Lie algebra generated by finitely many elements $a_1,a_2,\dots,a_m$ such that each commutator in these generators is ad-nilpotent. If $L$ satisfies a polynomial identity, then $L$ is nilpotent.
\end{theorem}

An important criterion for a Lie algebra to satisfy a polynomial identity is the following theorem.

\begin{theorem}[Bahturin-Linchenko-Zaicev]\label{22} Assume that a finite 
group $A$ acts on a Lie algebra $L$ by automorphisms in such a manner that $C_L(A)$, the subalgebra formed by fixed elements, satisfies a polynomial identity. Assume further that the characteristic of the ground field is either 0 or prime to the order of $A$. Then $L$ satisfies a polynomial identity.
\end{theorem}

The above theorem was first proved by Bahturin and Zaicev in the case where $A$ is soluble  \cite{bz} and later extended by Linchenko to the general case \cite{li}. In the present paper we only require the case where $A$ is abelian.

Let $G$ be a (profinite) group. A series of subgroups $$G=G_1\geq G_2\geq\dots\eqno{(*)}$$ is called an $N$-series if it satisfies $[G_i,G_j]\leq G_{i+j}$ for all $i,j\geq 1$. Here and throughout the paper when dealing with a profinite group we consider only closed subgroups. Obviously any $N$-series is central, i.e. $G_i/G_{i+1}\leq Z(G/G_{i+1})$ for any $i$.  Let $p$ be a prime. An $N$-series is called $N_p$-series if $G_i^p\leq G_{pi}$ for all $i$. Given an $N$-series $(*)$, let $L^*(G)$ be the direct sum of the abelian groups $L_i^*=G_i/G_{i+1}$, written additively. Commutation in $G$ induces a binary operation $[,]$ in $L$. For homogeneous elements $xG_{i+1}\in L_i^*,yG_{j+1}\in L_j^*$ the operation is defined by $$[xG_{i+1},yG_{j+1}]=[x,y]G_{i+j+1}\in L_{i+j}^*$$ and extended to arbitrary elements of $L^*(G)$ by linearity. It is easy to check that the operation is well-defined and that $L^*(G)$ with the operations $+$ and $[,]$ is a Lie ring. If all quotients $G_i/G_{i+1}$ of an $N$-series $(*)$ have prime exponent $p$ then $L^*(G)$ can be viewed as a Lie algebra over $\mathbb F_p$, the field with $p$ elements. In the important  case where the series $(*)$ is the $p$-dimension central series (also known under the name of Zassenhaus-Jennings-Lazard series) of $G$ we write $L_p(G)$ for the subalgebra generated by the first homogeneous component $G_1/G_2$ in the associated Lie algebra over the field with $p$ elements. Observe that the  $p$-dimension central series is an $N_p$-series (see \cite[p.\ 250]{hube} for details).

Any automorphism of $G$ in the natural way induces an automorphism of $L^*(G)$. If $G$ is profinite and $\alpha$ is a coprime automorphism of $G$, then the subring (subalgebra) of fixed points of $\alpha$ in $L^*(G)$ is isomorphic with the Lie ring associated to the group $C_G(\alpha)$ via the series formed by intersections of $C_G(\alpha)$ with the terms of the series $(*)$ (see \cite{aaaa} for more details).

Let $w=w(x_1,x_2,\dots,x_k)$ be a group-word. Let $H$ be a subgroup of a group $G$ and $g_1,g_2,\dots,g_k\in G$. We say that the law $w\equiv1$ is satisfied on the cosets $g_1H,g_2H,\dots,g_kH$ if $w(g_1h_1,g_2h_2,\dots,g_kh_k)=1$ for all $h_1,h_2,\dots,h_k\in H$. Wilson and Zelmanov showed in \cite{wize} that if a profinite group $G$ has an open subgroup $H$ and elements $g_1,g_2,\dots,g_k$ such that the law $w\equiv1$ is satisfied on the cosets $g_1H,g_2H,\dots,g_kH$, then $L_p(G)$ satisfies a polynomial identity for each prime $p$. More precisely, the proof in \cite{wize} shows that whenever a profinite group $G$ has an open subgroup $H$ and elements $g_1,g_2,\dots,g_k$ such that the law $w\equiv1$ is satisfied on the cosets $g_1H,g_2H,\dots,g_kH$, the Lie algebra $L^*(G)$ satisfies a multilinear polynomial identity for any prime $p$ and any $N_p$-series $(*)$ in $G$.

\begin{lemma} \label{remlocf} For any locally nilpotent profinite group $G$ there exist a positive integer $n$, elements $g_1,g_2\in G$ and an open subgroup $H\leq G$ such that the law $[x,{}_n\,y]\equiv1$ is satisfied on the cosets $g_1H,g_2H$. 
\end{lemma}
\begin{proof} Since any finitely generated subgroup of $G$ is nilpotent, for every pair of elements $g_1,g_2$ there exists a positive number $j$ such that $[g_1,{}_jg_2]=1$. For each integer $i$ we set $$S_i=\{(x,y)\in G\times G:[x,{}_i\,y]=1\}.$$ Since the sets $S_i$ are closed in $G\times G$ and have union $G\times G$, by Baire category theorem \cite[p.\ 200]{ke} at least one of these sets has a non-empty interior. Therefore we can find an open subgroup $H$ in $G$, elements $g_1,g_2\in G$ and an integer $n$ with the required property.
\end{proof}
The following proposition is now straightforward.

\begin{proposition}\label{34} Assume that a finite group $A$ acts coprimely on a profinite group $G$ in such a manner that $C_G(A)$ is locally nilpotent. Then for each prime $p$ the Lie algebra $L_p(G)$ satisfies a multilinear polynomial identity.
\end{proposition}
\begin{proof} Let $L=L_p(G)$. In view of Theorem \ref{22} it is sufficient to show that $C_L(A)$ satisfies a polynomial identity. We know that $C_L(A)$ is isomorphic with the Lie algebra associated with the central series of $C_G(A)$ obtained by intersecting $C_G(A)$ with the $p$-dimension central series of $G$. Since $C_G(A)$ is locally nilpotent, Lemma \ref{remlocf} applies. Thus, the Wilson-Zelmanov result \cite[Theorem 1]{wize} tells us that $C_L(A)$ satisfies a polynomial identity.
\end{proof}

We will also require the following lemma that essentially is due to Wilson and Zelmanov (cf \cite[Lemma in Section 3]{wize}.
\begin{lemma}\label{77} Let $G$ be a profinite group and $g\in G$ an element such that for any $x\in G$ there exists a positive $n$ with the property that $[x,{}_n\,g]=1$. Let $L^*(G)$ be the Lie algebra associated with $G$ using an $N_p$-series $(*)$ for some prime $p$. Then the image of $g$ in $L^*(G)$ is ad-nilpotent.
\end{lemma}
Finally, we quote a useful lemma from \cite{khushu}.

\begin{lemma}\label{khuz}
Let $L$ be a Lie algebra  and $H$ a subalgebra of $L$ generated by $m$ elements $h_1,\ldots,h_m$ such that all commutators in the generators $h_i$ are ad-nilpotent in $L$. If $H$ is nilpotent, then we have $[L,\underbrace{H,\ldots,H}_d]=0$ for some number $d$.
\end{lemma}

\section{Proof}

As usual, for a profinite group $G$ we denote by $\pi(G)$ the set of prime divisors of the orders of finite continuous homomorphic images of $G$. We say that $G$ is a $\pi$-group if $\pi(G)\subseteq\pi$ and $G$ is a $\pi'$-group if $\pi(G)\cap\pi=\emptyset$. If $m$ is an integer, we denote by $\pi(m)$ the set of prime divisors of $m$. If $\pi$ is a set of primes, we denote by $O_\pi(G)$ the maximal normal $\pi$-subgroup of $G$ and by $O_{\pi'}(G)$ the maximal normal $\pi'$-subgroup.

We are ready to embark on the proof of Theorem \ref{main}. 

\begin{proof}[Proof of Theorem \ref{main}]
Recall that $q$ is a prime and $A$ an elementary abelian group of order $q^3$ acting coprimely on a profinite group $G$ in such a manner that $C_G(a)$ is locally nilpotent for all $a\in A^\#$. We wish to show that $G$ is locally nilpotent. In view of Ward's Theorem \ref{wa} the group $G$ is pronilpotent and therefore $G$ is the Cartesian product of its Sylow subgroups.

Choose $a\in A^\#$.  By Lemma \ref{remlocf}  $C_G(a)$ contains an open subgroup $H$ and elements $u,v$ such that for some $n$ the law $[x,{}_n\,y]\equiv1$ is satisfied on the cosets $uH,vH$. Let $[C_{G}(a):H]=m$ and let $\pi_1=\pi(m)$. Denote $O_{\pi_1'}(C_G(a))$ by $T$. Since  $T$ is isomorphic to the image of $H$ in $C_G(a)/O_{\pi_1}(C_G(a))$, it is easy to see that $T$ satisfies the law $[x,{}_n\,y]\equiv1$, that is, $T$ is $n$-Engel. By the result of Burns and Medvedev \cite{bm} the subgroup $T$ has a nilpotent normal subgroup $U$ such that $T/U$ has finite exponent, say $e$. Set $\pi_2=\pi(e)$. Of course, the sets $\pi_1$ and $\pi_2$ depend on the choice of $a\in A^\#$ so strictly speaking they should be denoted by $\pi_1(a)$ and $\pi_2(a)$. For each such choice let $\pi_a=\pi_1(a)\cup\pi_2(a)$. 

We repeat this argument for every $a\in A^\#$. Set $\pi=\cup_{a\in A^\#}\pi_a$ and $K=O_{\pi'}(G)$. Since all sets $\pi_1(a)$ and $\pi_2(a)$ are finite, so is $\pi$. The choice of the set $\pi$ guarantees that $C_K(a)$ is nilpotent for every $a\in A^\#$. Thus, by Theorem \ref{2001}, the subgroup $K$ is nilpotent. Let $p_1,p_2,\dots,p_r$ be the finitely many primes in $\pi$ and let $P_1,P_2,\dots,P_r$ be the corresponding Sylow subgroups of $G$. Then $G=P_1\times P_2\times\dots\times P_r\times K$ and therefore it is sufficient to show that each subgroup $P_i$ is locally nilpotent. Thus, from now on without loss of generality we assume that $G$ is a pro-$p$ group for some prime $p$. Since every finite subset of $G$ is contained in a finitely generated $A$-invariant subgroup, we can further assume that $G$ is finitely generated.

Let $A_1,A_2,\dots,A_s$ be the distinct maximal subgroups of $A$. We denote by $D_j=D_j(G)$ the terms of the $p$-dimension central series of $G$. Set $L=L_p(G)$ and $L_j=L\cap (D_j/D_{j+1})$, so that $L=\oplus L_j$. The group $A$ naturally acts on $L$. Since each subgroup $A_i$ is noncyclic, by Lemma \ref{76} we have $L=\sum_{a\in A_i^\#}C_L(a)$ for every $i\leq s$.

Let $L_{ij}=C_{L_j}(A_i)$. Again by Lemma \ref{76}, for any $j$ we have $$L_j=\sum\limits_{1\leq i\leq s}L_{ij}.$$ In view of Lemma \ref{zizi} for any $l\in L_{ij}$ there exists $x\in D_j\cap C_G(A_i)$ such that $l=xD_{j+1}$. Therefore, by Lemma \ref{77}, the element $l$ is ad-nilpotent in $C_L(a)$ for every $a\in A_i^\#$. Since $L=\sum_{a\in A_i^\#}C_L(a)$, we conclude that  $$\mbox{any element $l$ in } L_{ij} \mbox{ is ad-nilpotent in } L. \eqno (**)$$

Let $\omega$ be a primitive $q$th root of unity and $\overline L=L\otimes \mathbb F_p[\omega]$. We can view $\overline L$ both as a Lie algebra over $\mathbb F_p$ and that over $\mathbb F_p[\omega]$. It is natural to identify $L$ with the $\mathbb F_p$-sub\-algebra $L\otimes 1$ of $\overline L$. We note that if an element $x\in L$ is ad-nilpotent of index $r$, say, then the ``same'' element $x\otimes 1$ is ad-nilpotent in $\overline L$ of the same index $r$. Put $\overline{L_j}=L_j\otimes\mathbb F_p[\omega]$; then $\overline L=\left<\overline{L_1}\right>$, since $L=\left<L_1\right>$, and $\overline L$ is the direct sum of the homogeneous components $\overline{L_j}$.

The group $A$ acts naturally on $\overline L$, and  we have $\overline{L_{ij}}=C_{\overline{L_j}}(A_i)$, where $\overline{L_{ij}}=L_{ij}\otimes\mathbb F_p[\omega]$. Let us show that $$\mbox{any element } y\in\overline{L_{ij}} \mbox{ is ad-nilpotent in $\overline L$.}\eqno (***)$$

Since $\overline{L_{ij}}=L_{ij}\otimes\mathbb F_p[\omega]$, we can write $$y=x_0+\omega x_1+\omega ^2x_2+ \dots +\omega ^{q-2}x_{q-2}$$ for some $x_0,x_1,x_2,\dots,x_{q-2}\in L_{ij}$, so that each of the summands $\omega^tx_t$ is ad-nilpotent by {(**)}. Set $J=\langle x_0,\omega x_1,\dots,\omega^{q-2}x_{q-2}\rangle$. This is the subalgebra generated by $x_0,\omega x_1,\dots,\omega^{q-2}x_{q-2}$. Note that $J\subseteq C_{\overline L}(A_i)$. A commutator of weight $k$ in the elements $\omega^{t}x_t$ has the form $\omega^{s}x$ for some $x$ that belongs to $L_{im}$, where $m=kj$. By {(**)} the element $x$ is ad-nilpotent and so such a commutator must be ad-nilpotent.

Proposition \ref{34} tells us that the Lie algebra $L$ satisfies a multilinear polynomial identity. The multilinear identity is also satisfied in $\overline L$ and so it is satisfied in $J$. Hence by Theorem \ref{21} $J$ is nilpotent. Lemma \ref{khuz} now says that $[\overline L,\underbrace{J,\dots,J}_{d}]=0$ for some $d$. This establishes (***). 

Since $A$ is abelian and the ground field is now a splitting field for $A$, every $\overline{L_j}$ decomposes in the direct sum of common eigenspaces for $A$. In particular, $\overline{L_1}$ is spanned by finitely many common eigenvectors for $A$. Hence $\overline L$ is generated by finitely many common eigenvectors for $A$ from $\overline{L_1}$. Every common eigenspace is contained in the centralizer $C_{\overline L}(A_i)$ for some $i\leq s$, since $A$ is of order $q^3$. We also note that any commutator in common eigenvectors is again a common eigenvector. Thus, if $l_1,\dots,l_{r}\in\overline{ L_1}$ are common eigenvectors for $A$ generating $\overline L$ then any commutator in these generators belongs to some $\overline{L_{ij}}$ and therefore, by (***), is ad-nilpotent.

As we have seen,  $\overline L$ satisfies a polynomial identity. It follows from Theorem \ref{21} that $\overline L$ is nilpotent. We now deduce that $L$ is nilpotent as well.

According to Lazard \cite{la} the nilpotency of $L$ is equivalent to $G$ being $p$-adic analytic. The Lubotzky-Mann theory \cite{luma} now tells us that $G$ is of finite rank, that is, all closed subgroups of $G$ are finitely generated. In particular, we conclude that $C_G(a)$ is finitely generated for every $a\in A^\#$. It follows that the centralizers $C_G(a)$ are nilpotent. Theorem \ref{2001} now tells us that $G$ is nilpotent. The proof is complete.
\end{proof}

\end{document}